\documentclass[11pt,reqno]{amsart}
\usepackage{mathrsfs}
\usepackage{amssymb}
\usepackage{amsfonts}
\usepackage{amsmath}
\usepackage{txfonts}
\usepackage{dsfont}
\usepackage{amscd}
\usepackage[colorlinks, linkcolor=blue, anchorcolor=blue, citecolor=blue]{hyperref}
\begin{document}
\setlength{\oddsidemargin}{0cm} \setlength{\evensidemargin}{0cm}
\baselineskip=20pt

\theoremstyle{plain} \makeatletter
\newtheorem{theorem}{Theorem}[section]
\newtheorem{Proposition}[theorem]{Proposition}
\newtheorem{Lemma}[theorem]{Lemma}
\newtheorem{Corollary}[theorem]{Corollary}
\newtheorem{Question}[theorem]{Question}

\theoremstyle{definition}
\newtheorem{notation}[theorem]{Notation}
\newtheorem{exam}[theorem]{Example}
\newtheorem{proposition}[theorem]{Proposition}
\newtheorem{conj}{Conjecture}
\newtheorem{prob}[theorem]{Problem}
\newtheorem{remark}[theorem]{Remark}
\newtheorem{claim}{Claim}
\newtheorem{Definition}[theorem]{Definition}

\newcommand{\SO}{{\mathrm S}{\mathrm O}}
\newcommand{\SU}{{\mathrm S}{\mathrm U}}
\newcommand{\Sp}{{\mathrm S}{\mathrm p}}
\newcommand{\so}{{\mathfrak s}{\mathfrak o}}
\newcommand{\Ad}{{\mathrm A}{\mathrm d}}
\newcommand{\m}{{\mathfrak m}}
\newcommand{\g}{{\mathfrak g}}
\newcommand{\h}{{\mathfrak h}}


\numberwithin{equation}{section}
\title[The moment map for the variety of  associative algebras]{The moment map for the variety of  associative algebras}

\author{Hui Zhang}
\address [Hui Zhang]{School of Mathematics, Southeast University, Nanjing 210096, P. R. China}\email{2120160023@mail.nankai.edu.cn}
\author{Zaili Yan}
\address{ School of Mathematics and Statistics, Ningbo University, Ningbo, Zhejiang Province, 315211,  People's Republic of China}\email{yanzaili@nbu.edu.cn}
\thanks{This work  is supported by NSFC (Nos. 11701300, 11626134) and  K.C. Wong Magna Fund in Ningbo University.}

\begin{abstract}
We consider the moment map $m:\mathbb{P}V_n\rightarrow \textnormal{i}\mathfrak{u}(n)$ for the action of $\textnormal{GL}(n)$ on $V_n=\otimes^{2}(\mathbb{C}^{n})^{*}\otimes\mathbb{C}^{n}$, and   study the critical points of  the functional  $F_n=\|m\|^{2}: \mathbb{P} V_n \rightarrow \mathbb{R}$.
Firstly, we   prove    that  $[\mu]\in \mathbb{P}V_n$  is   a critical point if and only if    $\textnormal{M}_{\mu}=c_{\mu} I+D_{\mu}$ for some $c_{\mu} \in \mathbb{R}$  and $D_{\mu} \in \textnormal{Der}(\mu),$ where $m([\mu])=\frac{\textnormal{M}_{\mu}}{\|\mu\|^{2}}$.
Then  we show that any algebra $\mu$ admits a Nikolayevsky derivation $\phi_\mu$ which is unique up to automorphism, and  if moreover,  $[\mu]$ is a critical point of $F_n$, then $\phi_\mu=-\frac{1}{c_\mu}D_\mu.$
Secondly, we  characterize   the  maxima and minima of the functional  $F_n: \mathcal{A}_n \rightarrow \mathbb{R}$,   where $\mathcal{A}_n$ denotes the projectivization of the algebraic varieties  of  all $n$-dimensional  associative algebras. Furthermore, for an arbitrary  critical point  $[\mu]$ of $F_n: \mathcal{A}_n \rightarrow \mathbb{R}$,  we also obtain a description  of  the algebraic structure of $[\mu]$.  Finally,   we classify the critical  points of $F_n: \mathcal{A}_n \rightarrow \mathbb{R}$  for $n=2$, $3$, respectively.
\end{abstract}

\maketitle

\section{Introduction}
Lauret  has studied  the moment map for the variety of Lie algebras    and   obtained  many remarkable   results in \cite{Lauret03},
which   turned out to be very  important   in proving that every Einstein solvmanifold is standard (\cite{Lauret2010}) and
in the characterization of solitons (\cite{BL2018,Lauret2011}).    Apart from the Lie algebras,   the study of the moment map in other classes of algebras was  also initiated  by Lauret, see \cite{Lauret2020} for more details. Motivated by this,  the authors have recently extended  the study of the moment map to the variety of    $3$-Lie algebras (see \cite{ZCL}).

In this paper,  we  study  the moment map for the variety of  {associative algebras}.  Let $\textnormal{GL}(n)$ be  the complex reductive  Lie group acting  naturally  on   the complex  vector  space $V_n=\otimes^{2}(\mathbb{C}^{n})^{*}\otimes\mathbb{C}^{n}$, i.e.,  the space of all $n$-dimensional complex algebras. The usual Hermitian inner product on $\mathbb{C}^{n}$   naturally induces an $\textnormal{U}(n)$-invariant  Hermitian inner product on $V_n$, which is denoted by  $\langle\cdot,\cdot\rangle$.  Since $\mathfrak{gl}(n)=\mathfrak{u}(n)+\textnormal{i}\mathfrak{u}(n)$,  we may
define a function as follows
\begin{align*}
m: \mathbb{P} V_n \rightarrow \textnormal{i}\mathfrak{u}(n),\quad
(m([\mu]), A)=\frac{(\textnormal{d}\rho_\mu)_{e}A}{\|\mu\|^{2}}, \quad 0 \neq \mu \in V_n, ~~ A\in \textnormal{i}\mathfrak{u}(n),
\end{align*}
where $(\cdot,\cdot)$ is an $\textnormal{Ad}(\textnormal{U}(n))$-invariant real inner product on $\textnormal{i}\mathfrak{u}(n),$ and $\rho_\mu:\textnormal{GL}(n)\rightarrow\mathbb{R}$ is defined by $\rho_\mu(g)=\langle g.\mu, g.\mu\rangle$.
The function $m$ is the moment
map from symplectic geometry, corresponding to the Hamiltonian action  $\textnormal{U}(n)$ of $V_n$ on
the symplectic manifold $\mathbb{P}V_n$ (see \cite{Kirwan98,Ness1984}).
In this paper, we   study    the critical points of the functional   $F_n=\|m\|^{2}: \mathbb{P} V_n \rightarrow \mathbb{R}$, with  an emphasis on  the critical points that lie in   the projectivization of the
algebraic variety  of  all $n$-dimensional associative algebras $\mathcal{A}_n$.

The paper is organized as follows:
In Sect.~\ref{section2}, we recall some basic concepts and results of complex associative algebras.

In Sect.~\ref{section3}, we first give the explicit expression of the moment map $m:\mathbb{P} V_n \rightarrow \textnormal{i}\mathfrak{u}(n)$ in terms of $\textnormal{M}_{\mu}$, that is, $m([\mu])=\frac{\textnormal{M}_{\mu}}{\|\mu\|^{2}}$ for any $[\mu]\in \mathbb{P}V_n$.  Then we show that  $[\mu]\in \mathbb{P}V_n$  is    a critical point of $F_n$ if and only if    $\textnormal{M}_{\mu}=c_{\mu} I+D_{\mu}$ for some $c_{\mu} \in \mathbb{R}$ and $D_{\mu} \in \textnormal{Der}(\mu)$ (Thm.~\ref{MID}).

In  Sect.~\ref{section4},  we  first show that any  algebra $\mu\in V_n$ admits a Nikolayevsky derivation $\phi_\mu$ which is unique up to automorphism,   the  eigenvalues of $\phi_\mu$  are necessarily  rational, and moreover,  $\phi_\mu=-\frac{1}{c_\mu}D_\mu$ if $[\mu]$ is a critical point of $F_n$ (Thm.~\ref{Nik}).
Then we  study the extremal points of $F_n:\mathcal{A}_n\rightarrow \mathbb{R},$ proving that the minimum value is
attained at semisimple associative algebras (Thm.~\ref{x}),  and the maximum value at the direct sum of a two-dimensional commutative associative algebra with the trivial algebra  (Thm.~\ref{mx}). In the context of Lie algebras (\cite{Lauret03}), Lauret  proved  that any $\mu$ for which
there exists  $[\lambda] \in \textnormal{GL}(n).[\mu]$ such that all eigenvalues of $\textnormal{M}_{\lambda}$ are negative, must be
semisimple, and we prove   that this result  also holds for associative algebras (Remark~\ref{negativeDefinite}).
Besides, the structure for  an arbitrary critical point $[\mu]$ of $F_n:\mathcal{A}_n\rightarrow \mathbb{R}$ is   discussed  (Thm.~\ref{structure} and Thm.~\ref{converse}).

In Sect.~\ref{eq}, we classify the critical  points of $F_n: \mathcal{A}_n \rightarrow \mathbb{R}$  for $n=2$, $3$.
It shows that every two-dimensional   associative algebra is isomorphic to a critical point of $F_2;$ and there exists only one  three-dimensional   associative algebra which is not isomorphic to  any  critical point of  $F_3$.
Finally,  based on the discussion in previous sections,  we collect some natural and interesting questions.

\section{Preliminaries}\label{section2}
In this section,  we recall some basic definitions and results of associative algebras. The ambient field is always assumed to be the complex number field $\mathbb{C}$ unless otherwise  stated.

\begin{Definition}\label{l-s}
A vector space $\mathcal{A}$ over  $\mathbb{C}$ with a bilinear operation
$\mathcal{A}\times \mathcal{A}\rightarrow \mathcal{A}$, denoted by $(x,y)\mapsto xy$, is called an \textit{associative algebra}, if
$$
x(yz)=(xy)z
$$
for all $x,y,z\in \mathcal{A}.$
\end{Definition}

A \textit{derivation} of an associative algebra  $\mathcal{A}$ is  a linear transformation $D:\mathcal{A}\rightarrow\mathcal{A}$ satisfying
\begin{align*}
D(xy)=(Dx)y+x(Dy),
\end{align*}
for $x,y\in\mathcal{A}.$ It is easy to see that the set of all  derivations of $\mathcal{A}$  is a Lie algebra,  which is denoted   by $\textnormal{Der}(\mathcal{A}).$
A vector subspace $I$ of $\mathcal{A}$ is called an \textit{ideal} if $\mathcal{A}I,I\mathcal{A}\subset I.$

\begin{Definition}\label{c-a}
Let $\mathcal{A}$ be an associative algebra. The \textit{center} of $\mathcal{A}$    is defined by $C(\mathcal{A})=\{x\in\mathcal{A}:xy=yx,\forall y\in\mathcal{A}\}$. The \textit{annihilator} of $\mathcal{A}$ is defined by $\textnormal{ann}(\mathcal{A})=\{x\in\mathcal{A}:xy=yx=0,\forall y\in\mathcal{A}\}$.
\end{Definition}
Clearly, $C(\mathcal{A})$ is a subalgebra of $\mathcal{A}$, and $\textnormal{ann}(\mathcal{A})$ is an ideal of $\mathcal{A}$.


\begin{Definition}\label{solvable}
Let ${I}$ be an ideal of an    associative algebra. Then ${I}$ is called \textit{nilpotent},  if ${I}^{k}=0$ for some integer $k\geq 1$, where ${I}^{k}=\underbrace{I{\cdots}I{\cdots}I}_k$.
\end{Definition}
If $I,J$ are any two nilpotent ideals of an    associative algebra $\mathcal{A}$, then $I+J$ is also  a nilpotent ideal. So the maximum nilpotent  ideal of $\mathcal{A}$ is unique, which is called the $radical$  and denoted by ${N}(\mathcal{A}).$
\begin{remark}\label{radical}
Note that ${N}(\mathcal{A})$ coincides with the Jacobson radical of $\mathcal{A}$ since $\mathcal{A}$ is an  associative algebra over $\mathbb{C}$. Moreover, ${N}(\mathcal{A})=\{x\in\mathcal{A}:  xy,yx \textnormal{~are nilpotent elements for any}~ y\in \mathcal{A}\}$.
\end{remark}

\begin{Definition}
Let $\mathcal{A}$ be an associative algebra. If $\mathcal{A}$ has no ideals except itself and $0,$ we call $\mathcal{A}$ simple.
\end{Definition}

Denote by  $\mathbb{M}_n(\mathbb{C})$  the set   of all $n\times n$
complex square matrices, which  is clearly an associative algebra with respect to the
usual matrix addition and multiplication. In fact, $\mathbb{M}_n(\mathbb{C})$  is a simple associative algebra for any $n\geq 1.$ Moreover,
it follows from Wedderburn-Artin theorem that any finite-dimensional simple associative algebra over $\mathbb{C}$ is isomorphic to $\mathbb{M}_n(\mathbb{C})$ for some integer $n\geq 1$ (\cite{Pierce}).

An associative algebra $\mathcal{A}$ is called $semisimple$ if its radical  ${N}(\mathcal{A})$ is zero. The following theorem is well known.
\begin{theorem}[\cite{Pierce}]\label{semidecom}
An associative algebra over $\mathbb{C}$ is semisimple if and only if it is a direct sum of simple ideals. That is, a semisimple associative algebra is isomorphic to
$\mathbb{M}_{n_1}(\mathbb{C})\times \mathbb{M}_{n_2}(\mathbb{C})\times\cdots\times \mathbb{M}_{n_s}(\mathbb{C})$ for some positive integers $n_1,n_2,\cdots,n_s.$
\end{theorem}

\section{The moment map for  complex algebras}\label{section3}
Let $\mathbb{C}^{n}$ be  the $n$-dimensional  complex vector space and  $V_n=\otimes^{2}({\mathbb{C}^{n}})^{*}\otimes\mathbb{C}^{n}$ be the space of all complex $n$-dimensional  algebras. The natural action  of $\textnormal{GL}(n)=\textnormal{GL}(\mathbb{C}^{n})$ on  $V_n$ is given  by
\begin{align}\label{Gaction}
g.\mu(X,Y)=g\mu(g^{-1}X,g^{-1}Y),\quad g\in \textnormal{GL}(n), X,Y\in\mathbb{C}^{n}.
\end{align}
Clearly, $\textnormal{GL}(n).\mu$ is precisely the isomorphism class of $\mu$. Note that
\begin{align*}
\lim_{t\rightarrow\infty}g_t.\mu=0,  \quad g_t=tI\subset\textnormal{GL}(n),t>0,
\end{align*}
we see  that  $0$ lies in the boundary of the orbit $\textnormal{GL}(n).\mu$ for each $\mu \in V_n.$
 By differentiating  (\ref{Gaction}), we obtain the natural action $\mathfrak{gl}(n)$ on  $V_n$, i.e.,
\begin{align}\label{gaction}
A.\mu(X,Y)=A\mu(X,Y)-\mu(AX,Y)-\mu(X,AY), \quad A\in \mathfrak{gl}(n),\mu\in V_n.
\end{align}
It follows that $A.\mu=0$ if and only if $A\in\textnormal{Der}(\mu),$ where  $\textnormal{Der}(\mu)$  denotes the derivation  algebra of $\mu.$

Note that the usual Hermitian inner product on $\mathbb{C}^{n}$   gives an $\textnormal{U}(n)$-invariant  Hermitian inner product on $V_n$ as follows
\begin{align}\label{metric}
\langle\mu,\lambda\rangle=\sum_{i,j,k}\langle\mu(X_i,X_j),X_{k}\rangle\overline{\langle\lambda(X_i,X_j),X_{k}\rangle},\quad\quad\quad~\mu,\lambda\in  V_n,
\end{align}
where $\{X_1,X_2,\cdots, X_n\}$ is an arbitrary orthonormal basis of  $\mathbb{C}^{n}$. Let  $\mathfrak{u}(n)$ denote the Lie algebra of $\textnormal{U}(n),$  then  it is easy to see that $\mathfrak{gl}(n)=\mathfrak{u}(n)+\textnormal{i}\mathfrak{u}(n)$ decomposes into skew-Hermitian and Hermitian transformations of $V_n$, respectively. Moreover, there is an $\textnormal{Ad}(\textnormal{U}(n))$-invariant Hermitian inner product on $\mathfrak{gl}(n)$ given by
\begin{align}\label{product}
(A,B)=\operatorname{tr}AB^{*},~ A,B\in \mathfrak{gl}(n).
\end{align}
The moment map  from  symplectic geometry,  corresponding to the Hamiltonian action  of $\textnormal{U}(n)$ on the symplectic
manifold  $\mathbb{P}V_n$,  is defined as follows
\begin{align}\label{moment}
m: \mathbb{P} V_n \rightarrow \textnormal{i}\mathfrak{u}(n), \quad(m([\mu]), A)=\frac{(\textnormal{d}\rho_\mu)_{e}A}{\|\mu\|^{2}}, \quad 0 \neq \mu \in V_n, A\in \textnormal{i}\mathfrak{u}(n),
\end{align}
where $\rho_\mu:\textnormal{GL}(n)\rightarrow\mathbb{R}$ is given by $\rho_\mu(g)=\langle g.\mu, g.\mu\rangle$. Clearly, $(\textnormal{d}\rho_\mu)_{e}A=\langle A.\mu,\mu\rangle+\langle \mu,A.\mu\rangle=2\langle A.\mu,\mu\rangle$ for any $A\in \textnormal{i}\mathfrak{u}(n)$.
The  square norm of the moment map is denoted by
\begin{align}\label{squarenorm}
F_n: \mathbb{P} V_n \rightarrow \mathbb{R},
\end{align}
where $F_n([\mu])=\|m([\mu])\|^{2}=(m([\mu]),m([\mu]))$ for any $[\mu]\in \mathbb{P} V_n.$

In order to express the moment map  $m$ explicitly,   we define $\textnormal{M}_{\mu}\in\textnormal{i}\mathfrak{u}(n)$ as follows
\begin{align}\label{M}
\textnormal{M}_{\mu}=2\sum_{i}{L}^{\mu}_{X_i}({L}^{\mu}_{X_i})^{*}-2\sum_{i}({L}^{\mu}_{X_i})^{*}{L}^{\mu}_{X_i}-2\sum_{i}({R}^{\mu}_{X_i})^{*}{R}^{\mu}_{X_i},
\end{align}
where  the left and right multiplication   ${L}^{\mu}_{X},{R}^{\mu}_{X}: \mathbb{C}^{n} \rightarrow \mathbb{C}^{n}$  by $X$ of the algebra $\mu$, are given  by ${L}^{\mu}_{X}(Y)=\mu(X,Y)$ and ${R}^{\mu}_{X}(Y)=\mu(Y,X)$ for all $Y\in\mathbb{C}^{n},$ respectively.   It is not hard  to prove that
\begin{align}\label{Mformula}
\langle\textnormal{M}_{\mu} X, Y\rangle=&2 \sum_{i,j} \overline{\langle\mu( X_i,X_j), X\rangle}\langle\mu(X_i,X_j), Y\rangle-2 \sum_{i,j}\langle\mu(X_i,X), X_{j}\rangle \overline{\langle\mu(X_i,Y), X_{j}\rangle} \notag\\
&-2 \sum_{i,j}\langle\mu(X, X_i), X_{j}\rangle \overline{\langle\mu(Y, X_i), X_{j}\rangle}
\end{align}
for $X,Y\in\mathbb{C}^{n}.$ Note that if  the algebra $\mu$ is commutative or anticommutative, then the second and  third term of (\ref{Mformula}) are the same, and in this case, $\textnormal{M}_{\mu}$ coincides with  \cite{Lauret03}.

\begin{Lemma}\label{Mm}
For  any  $\mu\in V_n$, we have  $(\textnormal{M}_{\mu}, A)=2\langle \mu,A.\mu\rangle$,  $\forall A\in \mathfrak{gl}(n) =\mathfrak{u}(n)+\textnormal{i}\mathfrak{u}(n).$
In particular,  $m([\mu])=\frac{\textnormal{M}_{\mu}}{\|\mu\|^{2}}$ for any $0\ne\mu\in V_n$
\end{Lemma}
\begin{proof}
For any $A\in \mathfrak{gl}(n)$, we have $(A,\textnormal{M}_{\mu})=\operatorname{tr}A\textnormal{M}_{\mu}^{*}=\operatorname{tr}A\textnormal{M}_{\mu}=\operatorname{tr}\textnormal{M}_{\mu}A,$
and
\begin{align*}
\operatorname{tr}\textnormal{M}_{\mu}A&=\underbrace{2\operatorname{tr}\sum_{i}{L}^{\mu}_{X_i}({L}^{\mu}_{X_i})^{*}A}_{\textnormal{I}}
-\underbrace{2\operatorname{tr}\sum_{i}(({L}^{\mu}_{X_i})^{*}{L}^{\mu}_{X_i}+({R}^{\mu}_{X_i})^{*}{R}^{\mu}_{X_i})A}_{\textnormal{II}}\\
 &=:\textnormal{I}-\textnormal{II}.
\end{align*}
Note that
\begin{align*}
\textnormal{I}
=&2\sum_{i}\operatorname{tr}{L}^{\mu}_{X_i}({L}^{\mu}_{X_i})^{*}A\\
=&2\sum_{i}\operatorname{tr}({L}^{\mu}_{X_i})^{*}A{L}^{\mu}_{X_i}\\
=&2\sum_{i,j}\langle({L}^{\mu}_{X_i})^{*}A{L}^{\mu}_{X_i}(X_j),X_{j}\rangle\\
=&2\sum_{i,j}\langle A\mu(X_{i},X_{j}),\mu(X_i,X_j)\rangle,
\end{align*}
and
\begin{align*}
\textnormal{II}=&2\operatorname{tr}\sum_{i}(({L}^{\mu}_{X_i})^{*}{L}^{\mu}_{X_i}+({R}^{\mu}_{X_i})^{*}{R}^{\mu}_{X_i})A\\
=&2\sum_{i,j}\langle(({L}^{\mu}_{X_i})^{*}{L}^{\mu}_{X_i}+({R}^{\mu}_{X_i})^{*}{R}^{\mu}_{X_i})AX_j,X_{j}\rangle\\
=&2\sum_{i,j}\langle{\mu}(X_{i},AX_{j}),{\mu}(X_{i},X_{j})\rangle-2\sum_{i,j}\langle{\mu}(AX_{j},X_{i}),{\mu}(X_{j},X_{i})\rangle\\
=&2\sum_{i,j}\langle \mu(AX_{i},X_{j})+{\mu}(X_{i},AX_{j}),\mu(X_{i},X_{j})\rangle.
\end{align*}
By (\ref{gaction}),  it follows that $(A,\textnormal{M}_{\mu})=\operatorname{tr}\textnormal{M}_{\mu}A=2\langle A.\mu,\mu\rangle,$ so $(\textnormal{M}_{\mu},A)=2\langle \mu, A.\mu\rangle$
for any $A\in \mathfrak{gl}(n)$. This proves the first statement.
For  $A\in \textnormal{i}\mathfrak{u}(n),$ we have $\langle A.\mu,\mu\rangle=\langle \mu,A.\mu\rangle$. By (\ref{moment}), we conclude that $m([\mu])=\frac{\textnormal{M}_{\mu}}{\|\mu\|^{2}}$ for any $0\ne\mu\in V_n$. This completes proof Lemma~\ref{Mm}
\end{proof}

\begin{Corollary}\label{MD}
For  any  $\mu\in V_n$, then
\begin{enumerate}
\item [(i)] $\operatorname{tr}\textnormal{M}_{\mu}D=0$ for any $D\in \textnormal{Der}(\mu);$
\item [(ii)] $\operatorname{tr}\textnormal{M}_{\mu}[A,A^{*}]\geq0$ for any $A\in \textnormal{Der}(\mu),$ and equality holds if and only if $A^{*}\in\textnormal{Der}(\mu).$
\end{enumerate}
\end{Corollary}
\begin{proof}
For \textnormal{(i)}, it follows from Lemma~\ref{Mm} that $\operatorname{tr}\textnormal{M}_{\mu}D=2\langle D.\mu,\mu\rangle.$ For \textnormal{(ii)}, it follows  from that $\operatorname{tr}\textnormal{M}_{\mu}[A,A^{*}]=2\langle [A,A^{*}].\mu,\mu\rangle=2\langle A^{*}.\mu,A^{*}.\mu\rangle\geq0$, $\forall A\in \textnormal{Der}(\mu)$, and the fact  $A^{*}.\mu=0$ if and only if $A^{*}\in\textnormal{Der}(\mu).$
\end{proof}

\begin{theorem}\label{MID}
The moment map $m: \mathbb{P} V_{n} \rightarrow \textnormal{i}\mathfrak{u}(n)$, the functional square norm of the moment map $F_{n}=\|m\|^{2}: \mathbb{P} V_{n} \rightarrow \mathbb{R}$ and the gradient of $F_{n}$ are, respectively, given by
\begin{align}\label{gradient}
F_{n}([\mu])=\frac{\operatorname{tr} \textnormal{M}_{\mu}^{2}}{\|\mu\|^{4}}, \quad \operatorname{grad}(F_{n})_{[\mu]}=\frac{ 8\pi_{*} (\textnormal{M}_{\mu}).\mu}{\|\mu\|^{4}}, \quad[\mu]\in \mathbb{P} V_{n},
\end{align}
where $\pi_{*}$ denotes the derivative of $\pi: V_{n} \backslash\{0\} \rightarrow \mathbb{P} V_{n}$, the canonical projection. Moreover, the following statements are equivalent:
\begin{enumerate}
\item [(i)] $[\mu] \in \mathbb{P} V_{n}$ is a critical point of $F_{n}$.
\item [(ii)] $[\mu] \in \mathbb{P} V_{n}$ is a critical point of $F_{n}|_{\textnormal{GL}(n).[\mu]}.$
\item [(iii)] $\textnormal{M}_{\mu}=c_{\mu} I+D_{\mu}$ for some $c_{\mu} \in \mathbb{R}$ and $D_{\mu} \in \textnormal{Der}(\mu)$.
\end{enumerate}
\end{theorem}
\begin{proof}
By (\ref{squarenorm}) and Lemma~\ref{Mm}, we have $F_{n}([\mu])=\frac{\operatorname{tr} \textnormal{M}_{\mu}^{2}}{\|\mu\|^{4}}$ for any $[\mu]\in \mathbb{P} V_{n}.$ To prove the second one, we only need to compute the gradient of $F_{n}: V_{n}\setminus\{0\} \rightarrow \mathbb{R}, ~F_{n}(\mu)=\frac{\operatorname{tr} \textnormal{M}_{\mu}^{2}}{\|\mu\|^{4}}$, and then to project it via $\pi_{*}.$
If $\mu, \lambda \in V_{n}$ with $\mu\ne 0$, then
\begin{align*}
\operatorname{Re}\langle\operatorname{grad}(F_{n})_{\mu}, \lambda\rangle
=&\left.\frac{\textnormal{d}}{\textnormal{d}t }\right|_{t=0} F_{n}(\mu+t\lambda)
=\left.\frac{\textnormal{d}}{\textnormal{d}t }\right|_{t=0} \frac{1}{\|\mu+t\lambda\|^{4}}( \textnormal{M}_{\mu+t\lambda},\textnormal{M}_{\mu+t\lambda})\\
=&-4\operatorname{Re}\langle \frac{F_{n}(\mu)}{\|\mu\|^{2}}\mu, \lambda\rangle
+\frac{2}{\|\mu\|^{4}}( \left.\frac{\textnormal{d}}{\textnormal{d}t}\right|_{t=0}\textnormal{M}_{\mu+t\lambda},\textnormal{M}_{\mu})
\end{align*}
We claim that
$(\left.\frac{\textnormal{d}}{\textnormal{d}t }\right|_{t=0}\textnormal{M}_{\mu+t\lambda},A)=4\operatorname{Re}\langle A.\mu,\lambda\rangle$ for any $A\in \textnormal{i}\mathfrak{u}(n).$
Indeed, by Lemma~\ref{Mm},  $(\left.\frac{\textnormal{d}}{\textnormal{d}t }\right|_{t=0}\textnormal{M}_{\mu+t\lambda},A)
=\left.\frac{\textnormal{d}}{\textnormal{d}t }\right|_{t=0}(\textnormal{M}_{\mu+t\lambda},A)
=2\left.\frac{\textnormal{d}}{\textnormal{d}t }\right|_{t=0}\langle \mu+t\lambda,A.(\mu+t\lambda)\rangle=2\langle \lambda,A.\mu\rangle+2\langle \mu,A.\lambda\rangle
=4\operatorname{Re}\langle A.\mu,\lambda\rangle$ for any $A\in \textnormal{i}\mathfrak{u}(n).$ It follows that $\operatorname{grad}(F_{n})_{\mu}=-4\frac{F_{n}(\mu)}{\|\mu\|^{2}}\mu+8\frac{(\textnormal{M}_{\mu}).\mu}{\|\mu\|^{4}},$ and consequentely
\begin{align*}
\operatorname{grad}(F_{n})_{[\mu]}=\frac{8\pi_{*} (\textnormal{M}_{\mu}).\mu}{\|\mu\|^{4}}.
\end{align*}
Thus the first part of the theorem is proved, and the following is to prove the equivalence among  the statements  $\textnormal{(i)},\textnormal{(ii)}$ and $\textnormal{(iii)}.$

$\textnormal{(i)}\Leftrightarrow\textnormal{(ii)}:$  The equivalence follows from that $\operatorname{grad}(F_{n})$ is tangent to the $\textnormal{GL}(n)$-orbits. Indeed
\begin{align*}
\operatorname{grad}(F_{n})_{[\mu]}=\frac{8\pi_{*} (\textnormal{M}_{\mu}).\mu}{\|\mu\|^{4}}=\frac{8}{\|\mu\|^{4}}\pi_{*}  (\left.\frac{\textnormal{d}}{\textnormal{d}t}\right|_{t=0}  e^{t\textnormal{M}_{\mu}}.\mu)
=\frac{8}{\|\mu\|^{4}}\left.\frac{\textnormal{d}}{\textnormal{d}t}\right|_{t=0}  e^{t\textnormal{M}_{\mu}}.[\mu]\in T_{[\mu]}(\textnormal{GL}(n).[\mu]).
\end{align*}

$\textnormal{(iii)}\Rightarrow\textnormal{(i)}:$ By (\ref{gaction}), we know that $I.\mu=-\mu,$  and $(\textnormal{M}_{\mu}).\mu=(c_\mu I+D_\mu).\mu=-c_\mu \mu.$ It follows that $\operatorname{grad}(F_{n})_{[\mu]}=0.$

$\textnormal{(i)}\Rightarrow\textnormal{(iii)}:$  Since $\operatorname{grad}(F_{n})_{[\mu]}=0,$ then $(\textnormal{M}_{\mu}).\mu\in\ker\pi_{*\mu}=\mathbb{C}\mu.$
 So $\textnormal{M}_{\mu}=c I+D$ for some $c\in \mathbb{C}$ and $D \in \textnormal{Der}(\mu).$ Clearly $[D,D^{*}]=[\textnormal{M}_{\mu}-cI,\textnormal{M}_{\mu}-\bar{c}I]=0,$ we conclude by Corollary~\ref{MD} that $D^{*}$ is also a derivation of $\mu.$ In particular, $(c-\bar{c})I=(D^{*}-D)\in \textnormal{Der}(\mu),$ thus $c=\bar{c}\in  \mathbb{R}.$
\end{proof}

\begin{remark}\label{crisum}
Let $[\mu]$ be a critical point of $F_{n}$ and $[\lambda]$ be a critical point of $F_{m}$, then $[\mu\oplus c\lambda]$ is a critical point of $F_{n+m}$ for a suitable $c\in\mathbb{C}.$  Indeed, assume that  $\textnormal{M}_{\mu}=c_{\mu} I+D_{\mu}$ for some $c_{\mu} \in \mathbb{R}$,  $D_{\mu} \in \textnormal{Der}(\mu)$, and $\textnormal{M}_{\lambda}=c_{\lambda} I+D_{\lambda}$ for some $c_{\lambda} \in \mathbb{R}$,  $D_{\lambda} \in \textnormal{Der}(\lambda)$. Noting that $\textnormal{M}_{t\lambda}=|t|^2\textnormal{M}_{\lambda}$ for any $t\in \mathbb{C}$, we can choose $t_0$ such that $c_\mu=|t_0|^2c_{\lambda}$,  then it follows that $[\mu\oplus t_0\lambda]$ is a critical point of $F_{n+m}.$
\end{remark}

In the frame of algebras, a remarkable result due to Ness can be stated as follows

\begin{theorem}[\cite{Ness1984}]\label{min_cri} If $[\mu]$ is a critical point of the functional $F_{n}: \mathbb{P} V_{n} \rightarrow \mathbb{R},$  then
\begin{enumerate}
\item [(i)] $\left.F_{n}\right|_{\mathrm{G} L(n) .[\mu]}$ attains its minimum value at $[\mu]$.
\item [(ii)] $[\lambda] \in \mathrm{GL} (n).[\mu]$ is a critical point of $F_{n}$ if and only if $[\lambda] \in \mathrm{U}(n).[\mu]$.
\end{enumerate}
\end{theorem}

In fact, the above theorem  implies that up to $\mathrm{U}(n)$-orbit, $\mathrm{GL} (n).[\mu]$ contains at most one critical point for each $[\mu]\in \mathbb{P} V_{n}.$

\begin{Lemma}\label{c}
Let $[\mu]\in\mathbb{P}V_n$ be a critical point of $F_n$ with $\textnormal{M}_{\mu}=c_{\mu} I+D_{\mu}$ for some $c_{\mu} \in \mathbb{R}$ and $D_{\mu} \in \textnormal{Der}(\mu)$. Then we have
\begin{enumerate}
\item [(i)] $c_\mu=\frac{\operatorname{tr}\textnormal{M}_{\mu}^{2}}{\operatorname{tr}\textnormal{M}_{\mu}}=-\frac{1}{2}\frac{\operatorname{tr}\textnormal{M}_{\mu}^{2}}{\|\mu\|^{2}}<0.$
\item [(ii)] If $\operatorname{tr}D_\mu\ne0$, then $c_\mu=-\frac{\operatorname{tr}{D}_{\mu}^{2}}{\operatorname{tr}{D}_{\mu}}$ and $\operatorname{tr}D_\mu>0.$
\end{enumerate}
\end{Lemma}
\begin{proof}
Since $\textnormal{M}_{\mu}=c_{\mu} I+D_{\mu}$, by  Lemma~\ref{Mm} and Corollary~\ref{MD} we have
\begin{align*}
\operatorname{tr}\textnormal{M}_\mu&=(\textnormal{M}_{\mu},I)=2\langle\mu,I.\mu \rangle=-2\|\mu\|^{2}<0,\\
\operatorname{tr}\textnormal{M}_\mu^{2}&=\operatorname{tr}\textnormal{M}_\mu (c_{\mu} I+D_{\mu})=c_\mu\operatorname{tr}\textnormal{M}_\mu.
\end{align*}
So $c_\mu=\frac{\operatorname{tr}\textnormal{M}_{\mu}^{2}}{\operatorname{tr}\textnormal{M}_{\mu}}=-\frac{1}{2}\frac{\operatorname{tr}\textnormal{M}_{\mu}^{2}}{\|\mu\|^{2}}<0.$ If $\operatorname{tr}D_\mu\ne0$, then
\begin{align*}
0=\operatorname{tr}\textnormal{M}_\mu D_{\mu}=c_\mu\operatorname{tr}D_\mu+\operatorname{tr}D_\mu^{2}.
\end{align*}
So $c_\mu=-\frac{\operatorname{tr}{D}_{\mu}^{2}}{\operatorname{tr}{D}_{\mu}}$ and $\operatorname{tr}D_\mu>0.$
\end{proof}

\begin{remark}\label{D=0}
In fact, $\operatorname{tr}D_\mu=0$ if and only if $D_\mu=0.$ Indeed, it follows from  that  $0=\operatorname{tr}\textnormal{M}_\mu D_\mu=c_\mu\operatorname{tr}D_\mu+\operatorname{tr}D_\mu^{2}$ and $D_\mu$ is Hermitian.
\end{remark}

\section{The critical points of the variety of   associative  algebras }\label{section4}
The space $\mathscr{A}_n$ of all $n$-dimensional  associative algebras   is an   algebraic set since it is given by polynomial conditions.
Denote by $\mathcal{A}_n$   the projective algebraic variety obtained by projectivization of $\mathscr{A}_n$ .  Note that $\mathcal{A}_n$  is $\textnormal{GL}(n)$-invariant,  then by Theorem~\ref{MID},  the critical points of $F_{n}: \mathcal{A}_n \rightarrow \mathbb{R}$ are precisely the critical points
of $F_n:\mathbb{P}V_n\rightarrow\mathbb{R}$ that lie in $\mathcal{A}_n$.

\subsection{The   Nikolayevsky derivation and the rationality}
A derivation of $\phi$ of an algebra $(\mu,\mathbb{C}^n)$ is called a \textit{Nikolayevsky derivation}, if it is semisimple with all eigenvalues real, and $\operatorname{tr}\phi \psi=\operatorname{tr}\psi$ for any $\psi\in \textnormal{Der}(\mu).$ This notion is motivated by \cite{Nik2011}.

\begin{theorem}\label{Nik}
Let $(\mu,\mathbb{C}^n)$ be an arbitrary  algebra. Then
\begin{enumerate}
\item [(1)] $(\mu,\mathbb{C}^n)$ admits  a  Nikolayevsky derivation $\phi_\mu$.
\item [(2)] The Nikolayevsky derivation $\phi_\mu$ is determined up to automorphism of $\mu$.
\item [(3)] All eigenvalues of $\phi_\mu$ are rational numbers.
\end{enumerate}
If moreover, $[\mu]$ is a critical point of $F_n:\mathbb{P}V_n\rightarrow\mathbb{R}$ with $\textnormal{M}_{\mu}=c_{\mu} I+D_{\mu}$ for some $c_{\mu} \in \mathbb{R}$ and $D_{\mu} \in \textnormal{Der}(\mu)$, then $-\frac{1}{c_\mu}D_\mu$ is the Nikolayevsky derivation of $\mu$.
\end{theorem}
\begin{proof}
(1) The complex Lie algebra $\textnormal{Der}(\mu)$ is algebraic. Let $\textnormal{Der}(\mu)=\mathfrak{s}\oplus\mathfrak{t}\oplus\mathfrak{n}$ be its Levi-Mal'cev decomposition, where $\mathfrak{s}$ is semisimple, $\mathfrak{t}\oplus\mathfrak{n}$ is the radical of $\textnormal{Der}(\mu)$,  $\mathfrak{n}$ is the set of all nilpotent elements in $\mathfrak{t}\oplus\mathfrak{n}$ (and is the nilradical of $\mathfrak{t}\oplus\mathfrak{n}$),  $\mathfrak{t}$ is an abelian subalgebra consisting of semisimple elements, and $[\mathfrak{s},\mathfrak{t}]=0.$
Define the bilinear form $B$ on $\textnormal{Der}(\mu)$ by $$B(\psi_1,\psi_2):=\operatorname{tr}\psi_1\psi_2,\quad \forall \psi_1,\psi_2\in \textnormal{Der}(\mu).$$  Then, in general, $B$ is degenerate, and $\operatorname{Ker}B=\mathfrak{n}.$ Since $\mathfrak{s}$ is semisimple, then  $B(\mathfrak{s},\mathfrak{t})=B([\mathfrak{s},\mathfrak{s}],\mathfrak{t})=B(\mathfrak{s},[\mathfrak{s},\mathfrak{t}])=0.$ Clearly, $B$ is nondegenerate on $\mathfrak{t}.$ Since  $\mathfrak{t}$ is reductive, we have $\mathfrak{t}=\mathfrak{a}+\textnormal{i}\mathfrak{a}$, where $\mathfrak{a}$ consists of semisimple elements with all eigenvalues real. It follows that there exists a unique element $\phi\in\mathfrak{a}$ such that $B(\phi, \psi)=\operatorname{tr}\psi$ for any $\psi\in \mathfrak{t}.$ Thus $\operatorname{tr}\phi \psi=\operatorname{tr}\psi$ for any $\psi\in \textnormal{Der}(\mu).$

(2) The subalgebra $\mathfrak{s}\oplus\mathfrak{t}$ is a maximal fully reducible subalgebra of $\textnormal{Der}(\mu).$ Since the maximal fully reducible subalgebras of $\textnormal{Der}(\mu)$ are conjugate by inner automorphism of $\textnormal{Der}(\mu)$ (which corresponds to an automorphism of $\mu$), and then the center $\mathfrak{t}$ of $\mathfrak{s}\oplus\mathfrak{t}$, is defined uniquely, up to automorphism.  So the Nikolayevsky derivation  is determined up to automorphism of $\mu$.

(3) The case $\phi_\mu=0$ is trivial. In the following, we assume that $\phi_\mu$ is nonzero. Note that $\phi_\mu$ is  simisimple with all eigenvalues real,  we have the following  decomposition
\begin{align*}
\mathbb{C}^{n}=\mathfrak{l}_1\oplus\mathfrak{l}_2\oplus\cdots\oplus\mathfrak{l}_r,
\end{align*}
where $\mathfrak{l}_i=\{X\in\mathbb{C}^{n}|\phi_\mu X=c_i X\}$ are eigenspaces of $\phi_\mu$ corresponding to eigenvalues   $c_1<c_2<\cdots<c_r\in\mathbb{R},$ respectively. Set $d_i=\dim\mathfrak{l}_i\in\mathbb{N},1\leq i\leq r.$ Since $\phi_\mu$ is a derivation, we have the following  relations
\begin{align*}
\mu(\mathfrak{l}_{i},\mathfrak{l}_{j})\subset\mathfrak{l}_{k}\quad \text{if}~ c_{i}+c_{j}=c_{k},
\end{align*}
for all $1\leq i,j,k\leq r.$ Conversely, if we define a linear transformation $\psi:\mathbb{C}^{n}\rightarrow\mathbb{C}^{n}$ by $\psi|_{\mathfrak{l}_i}=a_{i}\textnormal{Id}_{\mathfrak{l}_i}$, where $a_1,a_2,\cdots,a_r\in\mathbb{R}$ satisfy $a_i+a_j=a_k$ for all $1\leq i,j,k\leq r$ such that $c_i+c_j=c_k,$ then $\psi$ is a  derivation of $\mu.$ Clearly, all such derivations form  a real  vector space, which can be identified with
$W:=\{(w_1,w_2,\cdots,w_r)\in\mathbb{R}^{r}|w_{i}+w_{j}=w_{k}~\text{if}~ c_{i}+c_{j}=c_{k}\}.$  We endow $\mathbb{R}^{r}$ with the usual inner product, i.e.
\begin{align}\label{Rproduct}
\langle x,y\rangle=\sum_{i}x_iy_i,
\end{align}
for any $x=(x_1,x_2,\cdots,x_r),~y=(y_1,y_2,\cdots,y_r)\in\mathbb{R}^{r}.$

For any derivation $\psi\in W$, we have
\begin{align*}
\operatorname{tr}(\phi_\mu- I) \psi=\operatorname{tr}\phi_\mu\psi-\operatorname{tr}\psi=0.
\end{align*}
 Then we see that
$(d_1(c_1-1),d_2(c_2-1),\cdots,d_r(c_r-1))\perp W$ relative to (\ref{Rproduct}). Put $F:=W^{\perp}$, then by definition we have
\begin{align*}
F=\textnormal{span}_{1\leq i,j,k\leq r}\{e_{i}+e_{j}-e_{k}:c_{i}+c_{j}=c_{k}\},
\end{align*}
where $e_i$ belongs  to $\mathbb{R}^{r}$ having $1$ in the $i$-th position and $0$ elsewhere. Let $\{e_{i_{1}}+e_{j_{1}}-e_{k_{1}},\cdots,e_{i_{s}}+e_{j_{s}}-e_{k_{s}}\}$ be a basis of $F,$  then
\begin{align}\label{Rational}
(d_1(c_1-1),d_2(c_2-1),\cdots,d_r(c_r-1))=\sum_{p=1}^{s}b_p(e_{i_{p}}+e_{j_{p}}-e_{k_{p}}),
\end{align}
for some $b_1,b_2,\cdots,b_s\in \mathbb{R}.$
Put
\begin{align*}
E=\left(\begin{array}{c}
e_{i_{1}}+e_{j_{1}}-e_{k_{1}}\\
e_{i_{2}}+e_{j_{2}}-e_{k_{2}}\\
\vdots\\
e_{i_{s}}+e_{j_{s}}-e_{k_{s}}\\
\end{array}\right)\in \mathbb{Z}^{s\times r},
\end{align*}
then $EE^{T}\in \textnormal{GL}(s,\mathbb{Z}),$ and  $(EE^{T})^{-1}\in \textnormal{GL}(s,\mathbb{Q}).$ By (\ref{Rational}) and the definition of $E$, we have
\begin{align*}
\left(\begin{array}{c}
d_1(c_1-1)\\
d_2(c_2-1)\\
\vdots\\
d_r(c_r-1)
\end{array}\right)_{r\times 1}
=E^{T}\left(\begin{array}{c}
b_1\\
b_2\\
\vdots\\
b_s
\end{array}\right)_{s\times 1},~ E\left(\begin{array}{c}
c_1\\
c_2\\
\vdots\\
c_r
\end{array}\right)_{r\times 1}=\left(\begin{array}{c}
0\\
0\\
\vdots\\
0
\end{array}\right)_{s\times 1},~
\quad E\left(\begin{array}{c}
1\\
1\\
\vdots\\
1
\end{array}\right)_{r\times 1}=\left(\begin{array}{c}
1\\
1\\
\vdots\\
1
\end{array}\right)_{s\times 1}.
\end{align*}
By the left multiplication of $E$ on (\ref{Rational}), we have
\begin{align*}
\left(\begin{array}{c}
0\\
0\\
\vdots\\
0
\end{array}\right)_{s\times 1}-\left(\begin{array}{c}
1\\
1\\
\vdots\\
1
\end{array}\right)_{s\times 1}=ED^{-1}E^{T}\left(\begin{array}{c}
b_1\\
b_2\\
\vdots\\
b_s
\end{array}\right)_{s\times 1},
\end{align*}
where $D=\textnormal{diag}(d_1,d_2,\cdots,d_r).$ It is easy to see that  $(ED^{-1}E^{T})\in \textnormal{GL}(s,\mathbb{Q}).$ Consequently
\begin{align*}
D\left(\begin{array}{c}
c_1-1\\
c_2-1\\
\vdots\\
c_r-1
\end{array}\right)_{r\times 1}
=- E^{T}(ED^{-1}E^{T})^{-1}\left(\begin{array}{c}
1\\
1\\
\vdots\\
1
\end{array}\right)_{s\times 1},
\end{align*}
and
\begin{align*}
\left(\begin{array}{c}
c_1\\
c_2\\
\vdots\\
c_r
\end{array}\right)_{r\times 1}
=\left(\begin{array}{c}
1\\
1\\
\vdots\\
1
\end{array}\right)_{r\times 1}
- D^{-1}E^{T}(ED^{-1}E^{T})^{-1}\left(\begin{array}{c}
1\\
1\\
\vdots\\
1
\end{array}\right)_{s\times 1}\in \mathbb{Q}^{r}.
\end{align*}
So all eigenvalues of $\phi_\mu$ are  rational.

For the last statement, by Corollary~\ref{MD} we know that $0=\operatorname{tr}\textnormal{M}_{\mu}\psi=c_{\mu}\operatorname{tr}\psi+\operatorname{tr}D_\mu\psi$ for any $\psi\in\textnormal{Der}(\mu).$  Since $D_\mu$ is  Hermitian,    we conclude that $-\frac{1}{c_\mu}D_\mu$ is the Nikolayevsky derivation of $\mu$.
\end{proof}

By Theorem~\ref{Nik}, it is easy to obtain the following theorem.
\begin{theorem}\label{eigenvalue}
Let $[\mu]\in \mathbb{P}V_n$ be a critical point of $F_n:\mathbb{P}V_n\rightarrow\mathbb{R}$ with $\textnormal{M}_{\mu}=c_{\mu} I+D_{\mu}$ for some $c_{\mu} \in \mathbb{R}$ and $D_{\mu} \in \textnormal{Der}(\mu)$. Then there exists a constant $c>0$ such that the eigenvalues of $cD_\mu$ are  integers  prime to each other, say $k_1< k_2<\cdots < k_r \in \mathbb{Z}$ with multiplicities $d_1,d_2,\cdots,d_r\in \mathbb{N}.$
\end{theorem}

\begin{Definition}
The data set $(k_1<k_2<\cdots<k_r;d_1,d_2,\cdots,d_r)$ in \textnormal{Theorem}~\textnormal{\ref{eigenvalue}}  is called the
type of the critical point $[\mu].$
\end{Definition}

\begin{proposition}\label{CV}
Let $[\mu]\in\mathbb{P}V_n$ be a critical point of $F_n$ with type $\alpha=(k_1<k_2<\cdots<k_r;d_1,d_2,\cdots,d_r).$ Then we have
\begin{enumerate}
\item [(i)] If $\alpha=(0;n)$, then  $F_n ([\mu])=\frac{4}{n}.$
\item [(ii)] If $\alpha\ne(0;n)$, then  $F_n([\mu])=4\left(n-\frac{(k_1d_1+k_2d_2+\cdots+k_rd_r)^{2}}{k_1^{2}d_1+k_2^{2}d_2+\cdots+k_r^{2}d_r}\right)^{-1}.$
\end{enumerate}
\end{proposition}
\begin{proof}
We suppose that
$\textnormal{M}_{\mu}=c_{\mu} I+D_{\mu}, \|\mu\|=1.$
Since $\operatorname{tr}\textnormal{M}_{\mu}=-2\langle\mu , \mu\rangle=-2,$ then
\begin{align*}
\operatorname{tr}\textnormal{M}_{\mu}^{2}=\operatorname{tr}\textnormal{M}_{\mu}(c_\mu I+D_\mu)=c_\mu \operatorname{tr}\textnormal{M}_{\mu}=-2c_\mu,
\end{align*}
and
$
F_n ([\mu])=\frac{\operatorname{tr}{\textnormal{M}_{\mu}}^{2}}{\|\mu\|^{4}}=\operatorname{tr}{\textnormal{M}_{\mu}}^{2}=-2c_\mu.
$

For $\textnormal{(i)}$, we have $D_{\mu}=0$, so $\textnormal{M}_{\mu}=c_{\mu} I$ and $c_\mu n=\operatorname{tr}\textnormal{M}_{\mu}=-2.$ Thus $c_\mu =-\frac{2}{n}.$
$F_n ([\mu])=-2c_\mu=\frac{4}{n}.$

For $\textnormal{(ii)}$, we have $D_{\mu}\ne0$, and $c_\mu=-\frac{\operatorname{tr}{D}_{\mu}^{2}}{\operatorname{tr}{D}_{\mu}}$ by Lemma~\ref{c} and Remark~\ref{D=0}.  Note that
\begin{align*}
F_n ([\mu])=\operatorname{tr}{\textnormal{M}_{\mu}}^{2}=\operatorname{tr}(c_{\mu}I+D_\mu)^{2}=c_{\mu}^{2}n+c_{\mu}\operatorname{tr}D_\mu
=\frac{1}{4}F_n ([\mu])^{2} n-\frac{1}{2}F_n ([\mu])\operatorname{tr}D_\mu,
\end{align*}
so  we have
\begin{align*}
\frac{1}{F_n ([\mu])}=\frac{1}{4}n-\frac{1}{2F_n  ([\mu])}\operatorname{tr}(D_\mu)=\frac{1}{4}n+\frac{1}{4c_\mu}\operatorname{tr}D_\mu=\frac{1}{4}\left(n-\frac{(\operatorname{tr}D_\mu)^{2}}{\operatorname{tr}D_\mu^{2}}\right).
\end{align*}
It follows that $F_n([\mu])=4\left(n-\frac{(k_1d_1+k_2d_2+\cdots+k_rd_r)^{2}}{k_1^{2}d_1+k_2^{2}d_2+\cdots+k_r^{2}d_r}\right)^{-1}.$
\end{proof}

\subsection{The  minima of $F_{n}: \mathcal{A}_{n} \rightarrow \mathbb{R}$}
\begin{Lemma}\label{4/n}
Assume $[\mu]\in\mathbb{P}V_n,$ then $[\mu]$ is a critical point of $F_n:\mathbb{P}V_n\rightarrow\mathbb{R}$ with type $(0;n)$ if and only if $F_n([\mu])=\frac{4}{n}.$ Moreover, $\frac{4}{n}$ is the  minimum  value of $F_n:\mathbb{P}V_n\rightarrow\mathbb{R}.$
\end{Lemma}
\begin{proof}
For any $0\ne\mu\in V_n,$ we use $x_1,x_2,\cdots,x_n\in \mathbb{R}$ denote  the eigenvalues of $\textnormal{M}_{\mu}.$
Note that $\operatorname{tr}\textnormal{M}_{\mu}=-2\|\mu\|^{2}$, then we have
\begin{align*}
F_n ([\mu])=\frac{\operatorname{tr}{\textnormal{M}_{\mu}}^{2}}{\|\mu\|^{4}}=4\frac{\operatorname{tr}{\textnormal{M}_{\mu}}^{2}}{(\operatorname{tr}{\textnormal{M}_{\mu}})^{2}}
=4\frac{(x_1^{2}+x_2^{2}+\cdots +x_n^{2})}{(x_1+x_2+\cdots +x_n)^{2}}.
\end{align*}
It is easy  to see that $F_n ([\mu])\geq \frac{4}{n}$ with equality holds if and only if $x_1=x_2=\cdots=x_n.$ So $[\mu]$ is a critical point of $F_n:\mathbb{P}V_n\rightarrow\mathbb{R}$ with type $(0;n)$ if only if $\textnormal{M}_{\mu}$ is a constant multiple of $I$,  if and only $F_n$ attains its minimum value $\frac{4}{n}$ at $[\mu].$
\end{proof}

\begin{theorem}\label{x}
The functional  $F_{n}: \mathcal{A}_{n} \rightarrow \mathbb{R}$ attains its minimum value at a point $[\lambda] \in \textnormal{GL}(n).[\mu]$ if and only if $\mu$ is a semisimple associative algebra. In such a case, $F_{n}([\lambda])=\frac{4}{n}.$
\end{theorem}
\begin{proof}
Consider the simple associative  algebra $\mathbb{M}_{m}(\mathbb{C})$ for an integer $m>0.$  We endow $\mathbb{M}_{m}(\mathbb{C})$ with the following Hermitian inner product
\begin{align}\label{product}
\langle A,B\rangle:=\operatorname{tr}AB^{*},~ A,B\in \mathbb{M}_{m}(\mathbb{C}).
\end{align}
Then  $\{E_{ij}:1\leq i,j\leq m\}$ is an orthonormal basis,  where $E_{ij}$ denote the matrices having $1$ in the $(i,j)$-position and $0$ elsewhere. Set $\nu:=(\mathbb{M}_{m}(\mathbb{C}),\langle \cdot,\cdot\rangle)$. Clearly
$$
(L_A^\nu)^*=L_{A^*}, \quad (R_A^\nu)^*=R_{A^*}
$$
for any $ A\in \mathbb{M}_{m}(\mathbb{C}).$ Thus by (\ref{M}), we have
\begin{align*}
\textnormal{M}_{\nu}
&=2\sum_{ij}{L}^{\nu}_{E_{ij}}({L}^{\nu}_{E_{ij}})^{*}-2\sum_{ij}({L}^{\nu}_{E_{ij}})^{*}{L}^{\nu}_{E_{ij}}-2\sum_{ij}({R}^{\nu}_{E_{ij}})^{*}{R}^{\nu}_{E_{ij}}\\
&=2\sum_{ij}{L}^{\nu}_{E_{ij}}{L}^{\nu}_{E_{ji}}-2\sum_{ij}{L}^{\nu}_{E_{ji}}{L}^{\nu}_{E_{ij}}-2\sum_{ij}{R}^{\nu}_{E_{ji}}{R}^{\nu}_{E_{ij}}\\
&=2\sum_{ij}{L}^{\nu}_{E_{ij}E_{ji}}-2\sum_{ij}{L}^{\nu}_{E_{ji}E_{ij}}-2\sum_{ij}{R}^{\nu}_{E_{ij}E_{ji}}\\
&=2m\sum_{i}{L}^{\nu}_{E_{ii}}-2m\sum_{i}{L}^{\nu}_{E_{ii}}-2m\sum_{i}{R}^{\nu}_{E_{ii}}\\
&=2m{L}^{\nu}_{I}-2m{L}^{\nu}_{I}-2m{R}^{\nu}_{I}\\
&=2mI_{m^2}-2mI_{m^2}-2mI_{m^2}\\
&=-2mI_{m^2}.
\end{align*}
So $[\nu]$ is a critical point of type
$(0;m^2).$ Since $\mu$ is  a complex semisimple associative  algebra,   by Theorem~\ref{semidecom}, $\mu$ is isomorphic to
$\mathbb{M}_{n_1}(\mathbb{C})\times \mathbb{M}_{n_2}(\mathbb{C})\times\cdots\times \mathbb{M}_{n_s}(\mathbb{C})$ for some positive integers $n_1,n_2,\cdots,n_s.$ It follows from  Remark~\ref{crisum} that there exists  a point $[\lambda] \in \textnormal{GL}(n).[\mu]$  such that $[\lambda]$ is a critical point of type
$(0;n)$.  So  the functional $F_{n}: \mathcal{A}_{n} \rightarrow \mathbb{R}$ attains its minimum value at $[\lambda]$, and  $F_{n}([\lambda])=\frac{4}{n}$ by  Lemma~\ref{4/n}.

Conversely, assume that $F_{n}: \mathcal{A}_{n} \rightarrow \mathbb{R}$ attains its minimum value at a point $[\lambda] \in \textnormal{GL}(n).[\mu].$  The first part of the proof implies that $\textnormal{M}_\lambda=c_\lambda I$ with $c_\lambda<0.$
To prove $\mu$ is semisimple, it suffices to show   that $\mathcal{L}=(\lambda,\mathbb{C}^{n})$ is semisimple. Consider the following   orthogonal decompositions: (i) $\mathcal{L}=\mathcal{H}\oplus \mathcal{N}$, where $\mathcal{N}$ is the radical of $\lambda;$    (ii) $\mathcal{N}=\mathcal{V}\oplus\mathcal{Z}$, where  $\mathcal{Z}=\{A\in \mathcal{N}:\lambda(A,\mathcal{N})=\lambda(\mathcal{N},A)=0\}$ is the annihilator of $\mathcal{N}$. Clearly, $\mathcal{Z}$ is an   ideal of $\mathcal{L}$.
We have $\mathcal{L}=\mathcal{H}\oplus\mathcal{V}\oplus\mathcal{Z}.$
Suppose that $\mathcal{Z}\ne0.$ Let $\{H_i\},\{V_i\},\{Z_i\}$  be an orthonormal basis of $\mathcal{H},\mathcal{V},$ and $\mathcal{Z},$ respectively. Put $\{X_i\}=\{H_i\}\cup\{V_i\}\cup\{Z_i\}.$ For any $0\ne Z\in \mathcal{Z}$, by hypothesis we have
\begin{align*}
0>\langle\textnormal{M}_{\lambda} Z, Z\rangle
=&2 \sum_{i j}|\langle\lambda(X_{i}, X_{j}), Z\rangle|^{2}-2\sum_{ij}|\langle\lambda(Z, X_{i}), X_{j}\rangle|^{2}
-2\sum_{ij}|\langle\lambda(X_{i},Z), X_{j}\rangle|^{2} \\
=&2 \sum_{ij}\left\{|\langle\lambda(Z_i, H_{j}), Z\rangle|^{2}+|\langle\lambda(H_{i},Z_j), Z\rangle|^{2}
\right\}+\alpha{(Z)}\\
&-2\sum_{ij}|\langle\lambda(Z, H_{i}), Z_{j}\rangle|^{2}
-2\sum_{ij}|\langle\lambda(H_{i},Z), Z_{j}\rangle|^{2},
\end{align*}
where $\alpha{(Z)}=2\sum_{ij}|\langle\lambda(Y_{i}, Y_{j}), Z\rangle|^{2}\geq 0,$ $\{Y_i\}=\{H_i\}\cup\{V_i\}.$ This implies
\begin{align*}
0>\sum_{k}\langle\textnormal{M}_{\lambda} Z_k, Z_k\rangle=\sum_{k}\alpha{(Z_k)}\geq 0,
\end{align*}
which is a contradiction. So $\mathcal{Z}=0$,  and consequently, $\mathcal{N}=0.$  Therefore $\mathcal{L}$ is a  semisimple  associative algebra.

This completes the proof of theorem.
\end{proof}


\begin{remark}\label{negativeDefinite}
In fact, by the proof of Theorem~\ref{x}, we know that if  $[\mu]\in \mathcal{A}_n$ for which
there exists  $[\lambda] \in \textnormal{GL}(n).[\mu]$ such that $\textnormal{M}_{\lambda}$ is  negative definite, then $\mu $ is a semisimple  associative algebra.
\end{remark}

\subsection{The    maxima  of $F_{n}: \mathcal{A}_{n} \rightarrow \mathbb{R}$}
We say that an algebra $\lambda$ degenerates to $\mu$, write as $\lambda\rightarrow \mu$  if $\mu\in\overline{\textnormal{GL}(n).\lambda}$, the closure of $\textnormal{GL}(n).\lambda$ with respect to the usual topology of $V_n$. The degeneration $\lambda\rightarrow \mu$ is called \textit{direct  degeneration} if there are no nontrivial chains: $\lambda\rightarrow\nu\rightarrow \mu.$ The \textit{degeneration level} of an algebra is the maximum length of chain of  direct  degenerations.

\begin{theorem}[\cite{KO2013}]
An $n$-dimensional associative algebra is of degeneration level one if and only if  it is isomorphic to one of the following
\begin{enumerate}
\item [(1)]$\mu_{l}$:  $\mu_{l}(X_1,X_i)=X_i,~ i=1,\cdots,n;$
\item [(2)]$\mu_{r}$:  $\mu_{r}(X_i,X_1)=X_i,~ i=1,\cdots,n;$
\item [(3)]$\mu_{ca}$:  $\mu_{s}(X_1,X_1)=X_2$,
\end{enumerate}
where $\{X_1,\cdots,X_n\}$ is a basis.
\end{theorem}

\begin{theorem}\label{mx}
The functional $F_{n}: \mathcal{A}_{n} \rightarrow \mathbb{R}$ attains its maximal value at a point $[\mu]\in L_n,$ $n\ge3$ if and only if $\mu$ is isomorphic to the
commutative associative   algebra $\mu_{ca}$. In such a case, $F_{n}([\mu])=20.$
\end{theorem}
\begin{proof}
Assume that $F_{n}: \mathcal{A}_{n} \rightarrow \mathbb{R}$ attains its maximal value at a point $[\mu]\in \mathcal{A}_n,$  $n\ge3.$ By Theorem~\ref{MID}, we know that $[\mu]$ is also a critical of   $F_{n}: \mathbb{P}V_{n} \rightarrow \mathbb{R}.$ Then it follows  Theorem~\ref{min_cri} that $F_{n}|_{\textnormal{GL}(n).[\mu]}$ also attains its minimum value at a point $[\mu]$ , consequently   $F_{n}|_{\textnormal{GL}.[\mu]}$ is a constant, so
\begin{align}\label{GLU}
\textnormal{GL}(n).[\mu]=\textnormal{U}(n).[\mu]
\end{align}
The  relation (\ref{GLU}) implies that  the only non-trivial degeneration  of $\mu$ is $0$ (\cite[Theorem 5.1]{Lauret2003} ), consequently the  degeneration level of $\mu$ is $1$.

It is easy to see that  the critical points   $[\mu_{l}]$, $[\mu_{r}]$ are both  of type $(0<1;1,n-1)$, and  $[\mu_{ca}]$ is of type $(3<5<6;1,n-2,1).$
By Proposition~\ref{CV}, we know
\begin{align*}
F_{n}([\mu_{ca}])=20>4=F_{n}([\mu_{l}])=F_{n}([\mu_{r}]).
\end{align*}
So the theorem is proved.
\end{proof}

\subsection{The structure for the critical points of $F_{n}: \mathcal{A}_{n} \rightarrow \mathbb{R}$}   In the following,   we discuss  the structure for an arbitrary critical points of $F_{n}: \mathcal{A}_{n} \rightarrow \mathbb{R}$ by   Theorem~\ref{eigenvalue}.
\begin{theorem}\label{structure}
Let $[\mu]$ be a critical point of $F_{n}: \mathcal{A}_{n} \rightarrow \mathbb{R}$  with $\textnormal{M}_\mu=c_\mu I+D_\mu$  of type $(k_1<\cdots<k_r;d_1,d_2,\cdots,d_r)$,  where $c_{\mu} \in \mathbb{R}$ and $D_{\mu} \in \textnormal{Der}(\mu)$. Consider the orthogonal decomposition
$$\mathbb{C}^n=\mathfrak{A}_{-}\oplus\mathfrak{A}_0\oplus\mathfrak{A}_{+},$$
where $\mathfrak{A}_{-},$ $\mathfrak{A}_0$ and $\mathfrak{A}_{+}$ denote the direct sum of eigenspaces of $D_\mu$  with eigenvalues  smaller than zero, equal to zero  and larger than zero, respectively. Then the following conditions hold:
\begin{enumerate}
\item [(i)] $\textnormal{ann}(\mu)\subset\mathfrak{A}_+,$ where  $\textnormal{ann}(\mu)$ is the annihilator of $\mu$
\item [(ii)] $\mathfrak{A}_+\subset N(\mu),$ where  $N(\mu)$ is the radical of $\mu$.
\item [(iii)] $\mathfrak{A}_{-}\subset (C(\mu)\cap N(\mu))\setminus\textnormal{ann}(\mu),$  where $C(\mu)$ is the center  of $\mu$.
\item [(iv)] $(L_A^{\mu}-R_A^{\mu})^{*}\in \textnormal{Der}(\mu)$ for any $A\in \mathfrak{A}_0$.
So the  induced Lie algebra of $\mathfrak{A}_0$ is reductive.
\end{enumerate}
\end{theorem}
\begin{proof}
For (i), assume that $X\in \textnormal{ann}(\mu)$ and $D_\mu X=cX$,  then  by (\ref{Mformula})
\begin{align*}
\langle\textnormal{M}_{\mu} X, X\rangle=2\sum_{i,j}|\langle\mu(X_i,X_j), X\rangle|^{2}\geq 0.
\end{align*}
Since $\textnormal{M}_{\mu}=c_\mu I+D_\mu$, then $0\leq \langle\textnormal{M}_{\mu} X, X\rangle=(c_\mu+c)\langle X, X\rangle$. It follows  from Lemma~\ref{c} that  $c\geq-c_\mu>0.$ This proves (i).

For (ii), it is an immediate consequence of (iii) by Remark~\ref{radical}. Now, we  prove (iii) as follows. Assume that  $D_\mu X=cX$ for some $c<0$.  Since
$c L_X^{\mu}=[D_\mu,L_X^{\mu}],$
$c R_X^{\mu}=[D_\mu,R_X^{\mu}],$
then
\begin{align*}
c\operatorname{tr} (L_X^{\mu}-R_X^{\mu})(L_X^{\mu}-R_X^{\mu})^{*}&=\operatorname{tr}[D_\mu,(L_X^{\mu}-R_X^{\mu})](L_X^{\mu}-R_X^{\mu})^{*}\\
&=\operatorname{tr}[\textnormal{M}_\mu,(L_X^{\mu}-R_X^{\mu})](L_X^{\mu}-R_X^{\mu})^{*}\\
&=\operatorname{tr}\textnormal{M}_\mu[(L_X^{\mu}-R_X^{\mu}),(L_X^{\mu}-R_X^{\mu})^{*}].
\end{align*}
Noting that  $(L_X^{\mu}-R_X^{\mu})\in\textnormal{Der}(\mu)$,   by Corollary~\ref{MD}  we have
\begin{align*}
c\operatorname{tr} (L_X^{\mu}-R_X^{\mu})(L_X^{\mu}-R_X^{\mu})^{*}\geq 0.
\end{align*}
It follows that  $(L_X^{\mu}-R_X^{\mu})=0$ since  $c<0$. So $X\in C(\mu).$  By Remark~\ref{radical}, it is easy to see that $X\in N(\mu)$.
Using (i), we conclude $\mathfrak{A}_{-}\subset (C(\mu)\cap N(\mu))\setminus\textnormal{ann}(\mu).$ This proves (iii).

For (iv), we first note that
\begin{align*}
[D_\mu,L_A^{\mu}]=L_{D_\mu A}^{\mu}, \quad
[D_\mu,R_A^{\mu}]=R_{D_\mu A}^{\mu},
\end{align*}
for any $A\in \mathfrak{A}.$
If $A\in \mathfrak{A}_0$, we have $[D_\mu,L_A^{\mu}]=[D_\mu,R_A^{\mu}]=0$, and  so
\begin{align*}
\operatorname{tr} \textnormal{M}_\mu[(L_A^{\mu}-R_A^{\mu}),(L_A^{\mu}-R_A^{\mu})^{*}]
&=\operatorname{tr}  (c_\mu I+ D_\mu)[(L_A^{\mu}-R_A^{\mu}),(L_A^{\mu}-R_A^{\mu})^{*}]\\
&=\operatorname{tr}  D_\mu[(L_A^{\mu}-R_A^{\mu}),(L_A^{\mu}-R_A^{\mu})^{*}]\\
&=\operatorname{tr}  [D_\mu,(L_A^{\mu}-R_A^{\mu})](L_A^{\mu}-R_A^{\mu})^{*}\\
&=0.
\end{align*}
By Corollary~\ref{MD}, it follows that $(L_A^{\mu}-R_A^{\mu})^{*}\in \textnormal{Der}(\mu)$ since $(L_A^{\mu}-R_A^{\mu})\in \textnormal{Der}(\mu)$.   This proves (iv).
\end{proof}

In the sequel, we give a description of the critical points in terms of those which are nilpotent.
Let $[\lambda]$ be a nilpotent  critical point of $F_{m}: \mathcal{A}_{m} \rightarrow \mathbb{R}$. Define
\begin{align*}
L(\lambda):&=\{\Phi\in\textnormal{End}(\mathbb{C}^m):\Phi(\lambda(X,Y))=\lambda(\Phi X,Y)\},\\
R(\lambda):&=\{\Psi\in\textnormal{End}(\mathbb{C}^m):\Psi(\lambda(X,Y))=\lambda(X,\Psi Y)\}.
\end{align*}
Moreover, we set
$\Gamma_l=\{\Phi\in L(\lambda): [\Phi,\Psi]=0, \forall  \Psi\in R(\lambda)\},$
$\Gamma_r=\{\Psi\in R(\lambda): [\Phi,\Psi]=0, \forall  \Phi\in L(\lambda)\}$, and
\begin{align*}
\Gamma(\lambda):&=\{(\Phi,\Psi)\in \Gamma_l\times \Gamma_r:\lambda(\cdot,\Phi (\cdot))=\lambda(\Psi(\cdot),\cdot)\}.
\end{align*}
For any $(\Phi_i,\Psi_i)\in \Gamma(\lambda)$, $i=1,2$, we  define $(\Phi_1,\Psi_1)(\Phi_2,\Psi_2):=(\Phi_1\Phi_2,\Psi_2\Psi_1).$ Then it follows that $\Gamma(\lambda)$ is an associative algebra.

\begin{Lemma}\label{semisimple}
Assume that $\mathcal{S}\subset\Gamma(\lambda)$ is a subalgebra such that $(\Phi^{*},\Psi^{*})\in \mathcal{S}$ for any $(\Phi,\Psi)\in \mathcal{S}$, then $\mathcal{S}$ is a semisimple associative algebra.
\end{Lemma}
\begin{proof}
Note that $\mathcal{S}$ is an associative  algebra of  matrices, which are closed under conjugate transpose. Define an Hermitian inner product on $\mathcal{S}$ by
 \begin{align*}
\langle H_1,H_2\rangle:=\operatorname{tr}{H_1H_2^*}=\operatorname{tr}{\Phi_1\Phi_2^*}+\operatorname{tr}{\Psi_1\Psi_2^*}, ~\forall H_i=(\Phi_i,\Psi_i)\in \mathcal{S},i=1,2.
\end{align*}
Then it follows that  $\langle HH_1,H_2\rangle=\langle H_1,H^{*}H_2\rangle$, $\langle H_1H,H_2\rangle=\langle H_1,H_2H^{*}\rangle$ for any $H,H_1,H_2\in \mathcal{S}.$
Let $I$ be an ideal in $\mathcal{S}$ and $I^{\perp}$ denote the orthogonal complement of $I$.
Then it is easy to see that $I^{\perp}$ is also an ideal of $\mathcal{S}.$ Let $\mathcal{S}=\mathcal{R}\oplus \mathcal{N}$, where $\mathcal{N}$ is the radical of $\mathcal{S}$ and $\mathcal{R}=\mathcal{N}^{\perp}$. It follows that $\mathcal{R}$ and $\mathcal{N}$ are both ideals of $\mathcal{S}.$ Moreover, $\mathcal{R}$ is semisimple, and $\mathcal{N}$ is the annihilator of $\mathcal{S}$ (by considering the derived series). Since $\mathcal{S}$ is an associative  algebra of  matrices which are closed under conjugate transpose, then $HH^*=0$ for any $H\in \mathcal{N}$, hence $H=0$.  So $\mathcal{N}=0$, and  $\mathcal{S}$ is semisimple.
\end{proof}

\begin{theorem}\label{converse}
Let $[\lambda]$ be a nilpotent  critical point of $F_{m}: \mathcal{A}_{m} \rightarrow \mathbb{R}$ with $\textnormal{M}_\lambda=c_\lambda I+D_\lambda$   of type $(k_2<\cdots<k_r;d_2,\cdots,d_r),$ where  $c_{\lambda} \in \mathbb{R}$ and $D_{\lambda} \in \textnormal{Der}(\lambda)$. Assume that
$\mathcal{S}\subset\Gamma(\lambda)$ is a subalgebra  of dimension $d_1$ such that $(\Phi^{*},\Psi^{*})\in \mathcal{S}$, $[D_\lambda,\Phi]=[D_\lambda,\Psi]=0$ for any $(\Phi,\Psi)\in \mathcal{S}.$
 Consider the following  semidirect sum
\begin{align*}
\mu=\mathcal{S}\ltimes\lambda,
\end{align*}
where
\begin{align*}
\mu((\Phi_1,\Psi_1)+X_1,(\Phi_2,\Psi_2)+X_2):=(\Phi_1\Phi_2,\Psi_2\Psi_1)+\Phi_1(X_2)+\Psi_2(X_1)+X_1X_2,
\end{align*}
for any $(\Phi_1,\Psi_1),(\Phi_2,\Psi_2)\in \mathcal{S}$, $X_1,X_2\in \mathbb{C}^m$. Then $\mu$ is an associative algebra. If we extend the Hermitian inner product on $\mathbb{C}^{m}$ by setting
\begin{align*}
 \langle H,K\rangle=-\frac{2}{c_\lambda}(\operatorname{tr}L^{\mathcal{S}}_{H}L^{\mathcal{S}}_{{K}^{*}}
 +\operatorname{tr}{H}{K}^{*}),  ~~H,K\in \mathcal{S},
\end{align*}
 then $[\mu]$ is a critical point  of type $(0,k_2<\cdots<k_r;d_1,d_2,\cdots,d_r)$ for the functional  $F_{n}: \mathcal{A}_{n} \rightarrow \mathbb{R}$,  where $n=d_1+m.$
\end{theorem}
\begin{proof}
For any $H=(\Phi,\Psi)\in\mathcal{S}$, we have
\begin{align*}
L^\mu_{H}=\left(\begin{array}{cc}
L^{\mathcal{S}}_{H}&0\\
0&\Phi
\end{array}\right),\quad
R^\mu_{H}=\left(\begin{array}{cc}
R^{\mathcal{S}}_{H}&0\\
0&\Psi
\end{array}\right),
\end{align*}
where $L^\mu_{H}$, $R^\mu_{H}$ (resp. $L^{\mathcal{S}}_{H}$, $R^{\mathcal{S}}_{H}$) denote the left and right multiplication by $H$ of the algebra $\mu$ (resp. $\mathcal{S}$), respectively.  By Lemma~\ref{semisimple}, we know that  $\mathcal{S}$ is a semisimple associative algebra.  Then it follows that there is an orthonormal basis $\{H_i=(\Phi_i,\Psi_i)\}\subset\mathcal{S}$ such that  ${\Phi_i}^*=-\Phi_i$,  ${\Psi_i}^*=-\Psi_i$, and  $L^\mu_{H_i}$, $R^\mu_{H_i}$ are skew-Hermitian for each $i.$  Let $\{H_i\}\cup\{X_i\}$ be an orthonormal basis of $\mathbb{C}^{n}=\mathcal{S}\oplus\mathbb{C}^{m}.$  Then for any $H=(\Phi,\Psi)\in \mathcal{S}$ and $ X\in\mathbb{C}^{m}$,  we have
\begin{align*}
\langle\textnormal{M}_{\mu} X, H\rangle&=-2 \sum_{i,j}\langle\mu(X_i,X), X_{j}\rangle \overline{\langle\mu(X_i,H), X_{j}\rangle}
-2 \sum_{i,j}\langle\mu(X, X_i), X_{j}\rangle \overline{\langle\mu(H, X_i), X_{j}\rangle}\\
&=-2 \sum_{i,j}\langle\lambda(X_i,X), X_{j}\rangle \overline{\langle \Psi(X_i), X_{j}\rangle}
-2 \sum_{i,j}\langle\lambda(X, X_i), X_{j}\rangle \overline{\langle\Phi(X_i), X_{j}\rangle}\\
&=-2 \sum_{i}\langle\lambda(X_i,X), \Psi(X_i)\rangle
-2 \sum_{i}\langle\lambda(X, X_i), \Phi(X_i)\rangle \\
&=-2\operatorname{tr}\Psi^{*}R^{\lambda}_X-2\operatorname{tr}\Phi^{*}L^{\lambda}_X\\
&=-2\operatorname{tr}R^{\lambda}_{\Psi^{*}(X)}-2\operatorname{tr}L^{\lambda}_{\Phi^{*}(X)}\\
&=0,
\end{align*}
where $L^\lambda_{X}$, $R^\lambda_{X}$  denote the left and right multiplication by $X$ of the algebra $\lambda$, respectively,
and the last two equalities follow  from that $\lambda$ is nilpotent  and  $(\Phi^{*},\Psi^{*})\in \mathcal{S}.$  Moreover, since ${\Phi_i}^*=-\Phi_i$,  ${\Psi_i}^*=-\Psi_i$ for each $i$, then
$[\Phi_i,{\Phi_i}^*]=0$, $[\Psi_i,{\Psi_i}^*]=0.$ So  by (\ref{Mformula}) we have
\begin{align*}
\langle\textnormal{M}_{\mu} X, Y\rangle
&=2\sum_{i,j} \overline{\langle\mu( H_i,X_j), X\rangle}\langle\mu(H_i,X_j), Y\rangle+2\sum_{i,j} \overline{\langle\mu( X_i,H_j), X\rangle}\langle\mu(X_i,H_j), Y\rangle\\
&\quad +2\sum_{i,j} \overline{\langle\mu( X_i,X_j), X\rangle}\langle\mu(X_i,X_j), Y\rangle
-2\sum_{i,j} \langle\mu( H_i,X), X_j\rangle\overline{\langle\mu(H_i,Y), X_j\rangle}
\\
&\quad -2\sum_{i,j} \langle\mu( X_i,X), X_j\rangle\overline{\langle\mu(X_i,Y),X_j\rangle}
-2\sum_{i,j} \langle\mu( X,H_i), X_j\rangle\overline{\langle\mu(Y,H_i),X_j\rangle}
\\
&\quad -2\sum_{i,j} \langle\mu( X,X_i), X_j\rangle\overline{\langle\mu(Y,X_i),X_j\rangle}\\
&=\langle\textnormal{M}_{\lambda} X, Y\rangle+2\sum_i\langle[\Phi_i,{\Phi_i}^*] (X, Y\rangle+2\sum_i\langle[\Psi_i,{\Psi_i}^*] (X), Y\rangle\\
&=\langle\textnormal{M}_{\lambda} X, Y\rangle,
\end{align*}
for any $X,Y\in\mathbb{C}^{m}$.
Therefore ${\textnormal{M}_{\mu}}|_{\mathbb{C}^{m}}=\textnormal{M}_{\lambda}=c_\lambda I+D_\lambda.$ On the other hand, noting that $L^\mu_{H_i}$  and $R^\mu_{H_i}$ are skew-Hermitian for each $i$, then for any $H=(\Phi,\Psi)\in \mathcal{S}$, we have
\begin{align*}
\langle\textnormal{M}_{\mu} H, H\rangle&=2 \sum_{i,j} \overline{\langle\mu( H_i,H_j), H\rangle}\langle\mu(H_i,H_j), H\rangle\\
&\quad-2 \sum_{i,j}\langle\mu(H_i,H), H_{j}\rangle \overline{\langle\mu(H_i,H), H_{j}\rangle}-2 \sum_{i,j}\langle\mu(X_i,H), X_{j}\rangle \overline{\langle\mu(X_i,H), X_{j}\rangle}\\
&\quad -2 \sum_{i,j}\langle\mu(H, H_i), H_{j}\rangle \overline{\langle\mu(H, H_i), H_{j}\rangle}
-2 \sum_{i,j}\langle\mu(H, X_i), X_{j}\rangle \overline{\langle\mu(H, X_i), X_{j}\rangle}\\
&=-2(\operatorname{tr}L^{\mathcal{S}}_{H}L^{\mathcal{S}}_{{H}^{*}}
 +\operatorname{tr}{\Phi\Phi^*}+\operatorname{tr}{\Psi\Psi^*})\\
&=-2(\operatorname{tr}L^{\mathcal{S}}_{H}L^{\mathcal{S}}_{{H}^{*}}
 +\operatorname{tr}{H}{H}^{*})\\
 &=c_\lambda \langle H, H\rangle.
\end{align*}
So   $\textnormal{M}_\mu=c_\mu I+D_\mu,$ where  $c_\mu=c_\lambda,$ and
\begin{align*}
D_\mu=\left( {\begin{array}{*{20}{c}}
	0&0\\
	0&D_\lambda\\
\end{array}} \right)\in\textnormal{Der}(\mu).
\end{align*}
This completes the proof.
\end{proof}

\begin{remark}
Let the notation be as Theorem~\ref{structure}. If $(L_A^{\mu})^{*}\in \{L_A^{\mu}:A\in\mathfrak{A}_0\}$ and $(R_A^{\mu})^{*}\in \{R_A^{\mu}:A\in\mathfrak{A}_0\}$ for any $A\in \mathfrak{A}_0.$ Then it follows from a similar proof of  Lemma~\ref{semisimple} that $\mathfrak{A}_0$ is a semisimple associative algebra. Moreover, the radical of $[\mu]$ corresponds to a critical point of type $(k_1<\cdots<\hat{k}_s<\cdots<k_r;d_1,\cdots,\hat{d}_s,\cdots,d_r)$ by Theorem~\ref{converse}, where $k_s=0$.
\end{remark}

\section{Examples}\label{eq}
In this section,  we classify the critical points of $F_{n}: \mathcal{A}_{n} \rightarrow \mathbb{R}$  for $n=2$ and $3$, respectively.
It shows that every $2$-dimensional   associative algebra is isomorphic to a critical point of $F_{2}$, and there exists only one $3$-dimensional   associative algebra which is not isomorphic to any critical point of $F_{3}$. Finally,  based on the discussion in previous sections,  we collect some natural and interesting questions.

For reader's convenience, we recall the notation  in \cite{FP09}.
Let $\{e_1,e_2,\cdots,e_n\}$ be a basis of $\mathbb{C}^{n}$. Define the bilinear maps $\psi_k^{i,j}:\mathbb{C}^{n}\times\mathbb{C}^{n}\rightarrow\mathbb{C}^{n}$ by
\begin{align*}
\psi_k^{i,j}(e_me_n)=\delta_m^{i}\delta_n^{j}e_k.
\end{align*}
It follows that any algebra can be expressed in the form $d=\sum_{ijk}c_{ij}^k\psi_k^{i,j}$, where $c_{ij}^k\in\mathbb{C}$ are the structure constants.

\subsection{Two-dimensional case}
The classification of two-dimensional associative algebras can be found in \cite[TABLE 1]{FP09}.
We give the  classification of  the critical points of $F_{2}: \mathcal{A}_{2} \rightarrow \mathbb{R}$  as follows.
\begin{align*}
&\text { TABLE I. }  \text {Two-dimensional associative algebras, critical types and critical values. }\\
&\begin{array}{llllc}
\hline \hline
  \text{Multiplication relation}\quad\quad\quad\quad &  \text{Critical type}\quad\quad\quad\quad&\text{Critical value} \\
\hline
\left\{ d_1=\psi_1^{1,1} \right.&(0<1;1,1) & 4 \\
\left\{ d_2=\psi_1^{1,1}+\psi_2^{1,2}\right.&(0<1;1,1) & 4 \\
\left\{ d_3=\psi_1^{1,1}+\psi_2^{2,1} \right.&  (0<1;1,1) & 4 \\
\left\{ d_4=\psi_1^{1,1}+\psi_2^{2,2} \right. & (0;2) & 2 \\
\left\{ d_5=\psi_2^{1,1} \right.& (1<2;1,1) & 20\\
\left\{ d_6=\psi_1^{1,1}+\psi_2^{1,2}+\psi_2^{2,1} \right.&(0<1;1,1) & 4\\
\hline \hline
\end{array}
\end{align*}
Indeed, endow these algebras with the Hermitian inner product $\langle\cdot,\cdot\rangle$   so  that  $\{e_1,e_{2}\}$ is an orthonormal  basis, then it is easy to obtain TABLE I.
For example,  the multiplication relation of $\mu:=(d_6,\langle\cdot,\cdot\rangle)$ is given by:  $e_1e_1=e_1,e_1e_2=e_2,e_2e_1=e_2$. With respect to the given orthonormal  basis $\{e_1,e_{2}\}$, the left and right multiplications of $\mu$ are represented by
\begin{align*}
L^\mu_{e_1}=\left(\begin{array}{cc}
1&0\\
0&1
\end{array}\right),\quad
L^\mu_{e_2}=\left(\begin{array}{cc}
0&0\\
1&0
\end{array}\right),\quad
R^\mu_{e_1}=\left(\begin{array}{cc}
1&0\\
0&1
\end{array}\right),\quad
R^\mu_{e_2}=\left(\begin{array}{cc}
0&0\\
1&0
\end{array}\right).
\end{align*}
It follows from  (\ref{M}) that
\begin{align*}
\textnormal{M}_\mu=\left(\begin{array}{cc}
-6&0\\
0&0
\end{array}\right)
\end{align*}
Set $c_\mu:=\frac{\operatorname{tr}\textnormal{M}_\mu^2}{\operatorname{tr}\textnormal{M}_\mu}$, then  $c_\mu=-6$. It follows that  $\textnormal{M}_\mu=c_\mu I+D_\mu,$ where
\begin{align*}
D_\mu=\left(\begin{array}{cc}
0&0\\
0&6
\end{array}\right)
\end{align*}
is clearly a derivation of $\mu.$  So $[\mu]$ is a critical point of $F_{2}: \mathcal{A}_{2} \rightarrow \mathbb{R}$ with the critical type $(0<1;1,1)$ and $F_2([\mu])=4.$

\subsection{Three-dimensional case}
The complete classification of three-dimensional associative algebras can be found in \cite[TABLE 2]{FP09}. The  following table gives the classification of  the critical points of $F_{3}: \mathcal{A}_{3} \rightarrow \mathbb{R}$.
\begin{align*}
&\text { TABLE II. }  \text {Three-dimensional associative algebras, critical types and critical values. }\\
&\begin{array}{llllc}
\hline \hline
  \text{Multiplication relation }\quad\quad\quad\quad &  \text{Critical type}\quad\quad\quad\quad&\text{Critical value} \\
\hline
\left\{ d_1=\psi_1^{1,1} \right.&(0<1;1,2) & 4 \\
\left\{ d_2=\psi_1^{1,1}+\psi_3^{2,2}\right.&(0<1<2;1,1,1) & \frac{10}{3} \\
\left\{ d_3=\psi_1^{1,1}+\psi_3^{1,3} \right.&  (0<1;1,2) & 4 \\
\left\{ d_4=\psi_1^{1,1}+\psi_3^{3,1} \right. & (0<1;1,2) & 4  \\
\left\{ d_5=\psi_1^{1,1}+\psi_3^{1,3}+\psi_3^{3,1} \right.& (0<1;1,2) & 4 \\
\left\{ d_6=\psi_1^{1,1}+\psi_3^{3,3} \right.&(0<1;2,1) & 2\\
\left\{ d_7=\psi_1^{1,1}+\psi_2^{2,1}+\psi_3^{1,3} \right.&(0<1;1,2) & 4 \\
\left\{ d_8=\psi_1^{1,1}+\psi_2^{2,1}+\psi_3^{3,1}  \right.&(0<1;1,2) & 4 \\
\left\{ d_9=\psi_1^{1,1}+\psi_2^{2,1}+\psi_3^{1,3}+\psi_3^{3,1} \right.&(0<1;1,2) & 4 \\
\left\{ d_{10}=\psi_1^{1,1}+\psi_2^{2,1}+\psi_3^{3,3} \right.&(0<1;2,1) & 2\\
\left\{ d_{11}=\psi_1^{1,1}+\psi_2^{2,2}+\psi_3^{2,3} \right.&(0<1;2,1) & 2\\
\left\{ d_{12}=\psi_1^{1,1}+\psi_2^{2,2}+\psi_3^{2,3}+\psi_3^{3,2} \right.&(0<1;2,1) & 2\\
\left\{ d_{13}=\psi_1^{1,1}+\psi_2^{2,2}+\psi_3^{2,3}+\psi_3^{3,1} \right.&(0<1;2,1) & 2\\
\left\{ d_{14}=\psi_1^{1,1}+\psi_2^{2,2}+\psi_3^{3,3} \right.&(0;3) & \frac{4}{3}\\
\left\{ d_{15}=\psi_2^{1,1} \right.&(3<5<6;1,1,1) & 20\\
\left\{ d_{16}=\psi_2^{1,1}+\psi_3^{1,2}+\psi_3^{2,1} \right.&(1<2<3;1,1,1) & \frac{20}{3}\\
\left\{ d_{17}=\psi_1^{1,1}+\psi_2^{1,1}+\psi_2^{1,2}+\psi_2^{2,1}+\psi_3^{1,3} \right.&(0<1;1,2) & 4\\
\left\{ d_{18}=\psi_1^{1,1}+\psi_2^{1,1}+\psi_2^{1,2}+\psi_2^{2,1}+\psi_3^{1,3}+\psi_3^{3,1}  \right.&(0<1;1,2) & 4\\
\left\{ d_{19}=\psi_3^{3,3}+\psi_2^{1,1}+\psi_1^{1,3}+\psi_1^{3,1}+\psi_2^{2,3}+\psi_2^{3,2} \right.&(0<1<2;1,1,1) & \frac{10}{3}\\
\left\{ d_{20}=\psi_1^{1,1}+\psi_2^{1,2}+\psi_3^{1,3} \right.&(0<1;1,2) & 4\\
\left\{ d_{21}=\psi_3^{1,1}+\psi_3^{1,2}-\psi_3^{2,1} \right.&- & -\\
\left\{ d_{22}=x\psi_3^{1,2}+y\psi_3^{2,1}\right.&(1<2;2,1) & 12\\
\hline \hline
\end{array}
\end{align*}
Indeed,  endow the algebras with the Hermitian inner product $\langle\cdot,\cdot\rangle$   so  that  $\{e_1,e_{2},e_3\}$ is an orthonormal  basis, it  is easy  to obtain all  cases  in  TABLE II except for  $d_2,d_{10},d_{11},d_{12},d_{13},d_{17},d_{18},d_{21}$.
For the cases $d_2,d_{10},d_{11},d_{12}$,  it follows from Remark~\ref{crisum} and TABLE I.
For the cases $d_{13},d_{17},d_{18}$, it follows from \cite{KSTT} that   $d_{13}\cong U_1^3$, $d_{17}\cong W_{10}^3$ and  $d_{18}\cong U_{0}^3$,  where $U_1^3,W_{10}^3$ and $ U_{0}^3$ are defined by
\begin{align*}
U_1^3: \quad &\psi_1^{1,1}+\psi_1^{3,3}+\psi_2^{1,2}+\psi_2^{2,1}+\psi_2^{2,3}+\psi_3^{1,3}+\psi_3^{3,1}-\psi_3^{3,2}.\\
W_{10}^3:\quad & \psi_1^{1,2}+\psi_1^{2,1}+\psi_2^{2,2}+\psi_3^{2,3}.\\
U_{0}^3: \quad &\psi_2^{1,1}+ \psi_2^{1,2}+ \psi_2^{2,1}+ \psi_3^{1,3}+ \psi_3^{3,1}.
\end{align*}
Endow $U_1^3,W_{10}^3$ and $ U_{0}^3$ with the Hermitian inner product $\langle\cdot,\cdot\rangle$   so  that  $\{e_1,e_2,e_3\}$  is an orthonormal  basis, then it is easy to obtain the corresponding critical types and values for  $d_{13},d_{17},d_{18}$.

In the sequel,  we  follow a similar procedure as in \cite{KOTT,TT2018} to classify all Hermitian inner products on   $d_{21},$  then  show that for any Hermitian inner product $\langle\cdot,\cdot\rangle$ on   $d_{21},$  $(d_{21},\langle\cdot,\cdot\rangle)$  cannot be a critical point of $F_3$.
First,  note that the multiplication relation of  $d_{21}$   is given as follows:
\begin{align*}
e_1e_1=e_3,\quad e_1e_2=e_3, \quad e_2e_1=-e_3.
\end{align*}
Denote by $\langle\cdot,\cdot\rangle_0$  the Hermitian inner product  on $d_{21}$ such that $\{e_1, e_2, e_3\}$ is orthonormal.
With respect to this  basis $\{e_1, e_2, e_3\}$,  the automorphism group of $d_{21}$  is given by
\begin{align}\label{auto}
\textnormal{Aut}(d_{21})=\left(\begin{array}{ccc}
a&0&0\\
b&a&0\\
c&d&a^2
\end{array}\right)\subset\textnormal{GL}(3,\mathbb{C}),
\end{align}
where $0\ne a\in\mathbb{C}$, and $b,c,d\in\mathbb{C}$ are arbitrary.

\begin{Lemma}\label{repre}
For any Hermitian inner product $\langle\cdot,\cdot\rangle$ on $d_{21}$, there exist $k>0$ and $\phi\in\textnormal{Aut}(d_{21})$ such that $\{a\phi e_1,\phi e_2,\phi e_3\}$ is  orthonormal  with respective to  $k\langle\cdot,\cdot\rangle$, where $a>0$
\end{Lemma}
\begin{proof}
It suffices to prove that
$$\mathfrak{U}=\{\textnormal{diag}(a,1,1):a>0\}\subset\textnormal{GL}(3,\mathbb{C})$$
is a set of representatives for the action $\mathbb{C}^{\times}\textnormal{Aut}(d_{21})$ on  $\mathfrak{M}$, i.e., the space of all Hermitian inner products on $d_{21}$, which can be  identified with the homogeneous space   $\textnormal{GL}(3,\mathbb{C})/\textnormal{U}(3)$ at the base point $\langle\cdot,\cdot\rangle_0\in \mathfrak{M}$ (see \cite{KOTT}). Indeed, since
$$\bigcup_{g\in \mathfrak{U}}\mathbb{C}^{\times}\textnormal{Aut}({d_{21}})\cdot g\cdot \textnormal{U}(3)= \textnormal{GL}(3,\mathbb{C}),$$
it follows that $\mathfrak{U}$ is  a set of representatives.  For any Hermitian inner product $\langle\cdot,\cdot\rangle$ on $d_{21}$, we know that there exists $g_0\in \mathfrak{U}$ such that
\begin{align*}
 \langle\cdot,\cdot\rangle\in (\mathbb{C}^{\times}\textnormal{Aut}({d_{21}})).(g_0.\langle\cdot,\cdot\rangle_0)
\end{align*}
Hence there exist $c\in\mathbb{C}^{\times}$, $\phi\in\textnormal{Aut}({d_{21}})$ such that
\begin{align*}
 \langle\cdot,\cdot\rangle= (c\phi).(g_0.\langle\cdot,\cdot\rangle_0)=(c\phi g_0).\langle\cdot,\cdot\rangle_0)
\end{align*}
Put $k=|c|^2$, then
\begin{align*}
 k\langle\cdot,\cdot\rangle= k\langle (c\phi g_0)^{-1}(\cdot),(c\phi g_0)^{-1}(\cdot)\rangle_0
 =kc^{-1}\bar{c}^{-1}\langle  (\phi g_0)^{-1}(\cdot), (\phi g_0)^{-1}(\cdot)\rangle_0
 =\langle  (\phi g_0)^{-1}(\cdot), (\phi g_0)^{-1}(\cdot)\rangle_0
\end{align*}
Since $g_0\in\mathfrak{U}$, then $g_0=\textnormal{diag}\{a,1,1\}$ for some $a>0$.  It follows that $\{a\phi e_1,\phi e_2,\phi e_3\}$ is  orthonormal  with respective to  $k\langle\cdot,\cdot\rangle$
\end{proof}

\begin{Proposition}
For any Hermitian inner product $\langle\cdot,\cdot\rangle$ on $d_{21}$, $(d_{21},\langle\cdot,\cdot\rangle)$ can  not be a critical point of $F_{3}: \mathcal{A}_{3} \rightarrow \mathbb{R}.$
\end{Proposition}
\begin{proof}
Assume that $\langle\cdot,\cdot\rangle$ is a Hermitian inner product on $d_{21}$ such that  $(d_{21},\langle\cdot,\cdot\rangle)$ is  a critical point of $F_{3}: \mathcal{A}_{3} \rightarrow \mathbb{R}.$ Then the critical type is necessarily of  $(1<2;2,1)$ by Theorem~\ref{Nik} and (\ref{auto}). Moreover, for the Hermitian inner product $\langle\cdot,\cdot\rangle$ on $d_{21}$,  by Lemma~\ref{repre} we know that there exist $k>0$ and $\phi\in\textnormal{Aut}(d_{21})$ such that $\{x_1=a\phi e_1,x_2=\phi e_2,x_3=\phi e_3\}$ is  orthonormal  with respective to  $k\langle\cdot,\cdot\rangle$, where $a>0.$ With respect to the basis $\{x_1,x_2,x_3\}$, the multiplication relation  of $d_{21}$ is given as follows
\begin{align*}
x_1x_1=a^2x_3,\quad x_1x_2=ax_3,\quad x_2x_1=-ax_3.
\end{align*}
By  (\ref{M}), Lemma~\ref{c} and a straightforward calculation, it follows that the critical type is of
$$(3a^4+6a^2+8,5a^4+10a^2+8,2(3a^4+8a^2+8))$$
which is never of  type $(1<2;2,1)$ for any $a>0.$ This is a contradiction by Theorem~\ref{Nik}, and  the proposition is therefore proved.
\end{proof}

\subsection{Comments}
By the  previous discussion,  we know that the critical types  of $F_{n}: \mathcal{A}_{n} \rightarrow \mathbb{R}$, $n=2,3$, are necessarily nonnegative.  So it is natural to ask the following question:
\textit{Let $[\mu]\in \mathcal{A}_{n}$ be a critical point of $F_{n}: \mathcal{A}_{n} \rightarrow \mathbb{R}$  with $\textnormal{M}_\mu=c_\mu I+D_\mu$ for some $c_{\mu} \in \mathbb{R}$ and $D_{\mu} \in \textnormal{Der}(\mu)$.  Are all the eigenvalues of   $D_\mu$  necessarily nonnegative?}

On the other hand,   it will be also interesting to construct  or classify  the critical points  $[\mu]$ of  $F_{n}: \mathcal{A}_{n} \rightarrow \mathbb{R}$ such that $D_\mu$ has negative eigenvalues if the  above question does not hold. We note that  $2$-step nilpotent Lie algebras are automatically  associative algebras,  so it follows from \cite[Example~1]{Nik2006} that there  exist associative algebras whose Nikolayevsky derivations do admit negative eigenvalues.

\section{Statements and Declarations}
The authors declare that  there is no conflict of interest.

%

\end{document}